\documentclass[a4paper,10pt]{article}
\usepackage[utf8]{inputenc}
\usepackage{mathrsfs,pgf,tikz,amsmath,amssymb,hyperref,amsthm,stmaryrd}
\usepackage{colortbl}
\usepackage{enumitem}
\usetikzlibrary{arrows}

\theoremstyle{plain}
\newtheorem{theorem}{Theorem}[section]                                          
\newtheorem{proposition}[theorem]{Proposition}                          

\newtheorem{corollary}[theorem]{Corollary}

\theoremstyle{definition}
\newtheorem{definition}[theorem]{Definition}
\theoremstyle{remark}
\newtheorem{remark}[theorem]{Remark}
\newtheorem{example}[theorem]{Example}

\DeclareMathOperator{\p}{P}

\title{On minimal Besicovitch arrangements}
\author{Blondeau Da Silva St\'{e}phane}

\begin{document}
 
\maketitle
 
\begin{abstract}
In this paper we focus on minimal Besicovitch arrangements to highlight some of their properties. An appropriate probability space enables us to find 
again in an elegant way some straightforward equalities associated with these arrangements. Resulting inequalities are also brought out. A connection 
with arrangements of lines in $\mathbb{R}^2$ is eventually made, where possible. 
\end{abstract}

\section*{Introduction}

Let ${\mathbb{F}_q}$ be the $q$ elements finite field where $q$ is a prime power. A Besicovitch arrangement $B$ is a set of lines that contains at least 
one line in each direction. A minimal Besicovitch arrangement is a Besicovitch arrangement that is the union of exactly $q+1$ lines in ${\mathbb{F}_q}^2$ 
(see \cite{ste}).

Blondeau Da Silva proved (in \cite{ste})that the expected value of the number of points of multiplicity $m$ in ${\mathbb{F}_q}^2$, for $0\le m\le q+1$, 
with respect to a randomly chosen arrangement of $q+1$ lines with different slopes: $\big(1/(m!e)\big)q^2+O(q)$, as $q\rightarrow\infty$. He also 
demonstrated that the distance between the number of points in such a randomly chosen arrangement and $\big(1/(m!e)\big)q^2$ is lower than $q\ln q$ 
with probability close to $1$ for large $q$. 

In the first section of the paper, we take advantage of the probability space defined by Faber (in \cite{Faber}) and also of certain specific random variables to 
identify some of the constraints that such an arrangement is subject to. Equalities and inequalities between the multiplicities of all points in 
${\mathbb{F}_q}^2$ determined by this arrangement will be brought to light. Other properties of these arrangements will be emphasised in the second 
section, allowing us to make a connection with quantities specific to arrangements of $q+1$ lines in $\mathbb{R}^2$. 

\section{Some constraints minimal Besicovitch arrangements are subject to}\label{2}

\subsection{The chosen probability space}

Let $q$ be a prime power and ${\mathbb{F}_q}$ be the $q$ elements finite field.

A line, in ${\mathbb{F}_q}^2$, is a one-dimensional affine subspace. For $s\in{\mathbb{F}_q}\cup\{\infty\}$, $b\in{\mathbb{F}_q}$, let $l(s,b)$ denote the 
line :
\[\left \{\begin{array}{l @{\qquad \text{if }\enspace} l}
    y=sx+b & s\in{\mathbb{F}_q}, \\
    x=b & s=\infty.
\end{array}
\right.\]

Recall from \cite{Faber} that a Besicovitch set in ${\mathbb{F}_q}^2$ is a set of points $E\subset{\mathbb{F}_q}^2$ such that:
\begin{center}
$\forall i\in{\mathbb{F}_q}\cup\{\infty\},\quad\exists b_i\in{\mathbb{F}_q}\quad\text{so that } l(i,b_i)\subset E.$
\end{center}

Let $B$ be a Besicovitch arrangement. Henceforth we denote by $\tilde{B}$ the Besicovitch set formed by all the points belonging to the lines of the 
Besicovitch arrangement $B$.

We define our probability space $\Omega$, as was done by Faber (in \cite{Faber}), \textit{i.e.}: $\Omega=\bigoplus_{i\in{\mathbb{F}_q}\cup\{\infty\}}{\mathbb{F}_q}$, 
assigning probability $q^{-(q+1)}$ to each element in $\Omega$. The associated probability law is denoted by $\p$. 

From now on, we identify each element $\sum_{i\in{\mathbb{F}_q}\cup\{\infty\}} b_i$ in $\Omega$ with the minimal Besicovitch arrangement 
$B=\bigcup_{i\in{\mathbb{F}_q}\cup\{\infty\}}l(i,b_i)$.

For a Besicovitch arrangement $B$, for $0\le m\le q+1$, we define $x_m^B$ as the number of points in ${\mathbb{F}_q}^2$ through which exactly $m$ lines 
pass, \textit{i.e.} the number of points of multiplicity $m$ in ${\mathbb{F}_q}^2$.

For $0\le m\le q+1$, let us denote by $X_m$ the random variable from $\Omega$ to $\mathbb{N}$ that maps each element $\sum_{i\in{\mathbb{F}_q}\cup
\{\infty\}} b_i$ to $x_m^B$, where $B=\bigcup_{i\in{\mathbb{F}_q}\cup\{\infty\}}l(i,b_i)$.

\subsection{The three equalities}

We can take advantage of the probability space defined above to bring out some equalities associated with minimal Besicovitch arrangements. 
To this end, we will reuse the notations of the proof of \cite[Theorem 1.]{ste}.\\
For $P$ in ${\mathbb{F}_q}^2$, let $M_{P}$ be the random variable that maps $B\in\Omega$ to the multiplicity of $P$ in $B$. For $0\le m\le q+1$, let 
$f_{m,P}:\Omega\rightarrow\mathbb{R}$ be the random variable defined by:
\[f_{m,P}(B)=\left \{
\begin{array}{c @{\quad} l}
    1 & \text{if }\enspace M_{P}(B)=m, \\
    0 & \text{otherwise.}
\end{array}
\right.\]
It follows that for $0\le m\le q+1$ and for $B\in\Omega$:
\begin{equation}\label{defX}
X_m(B)=\sum_{P\in{\mathbb{F}_q}^2}f_{m,P}(B).
\end{equation}

\subsubsection{The first equality}

We first consider the random variable $S$ that maps $B\in\Omega$ to $\frac{1}{q^2}\sum_{m=0}^{q+1}X_m(B)$. We have the following results:
\begin{proposition}\label{Esp''}
$E(S)=1\quad$ and $\quad Var(S)=0.$
\end{proposition}
\begin{proof}
We first compute the expected value of $S$:
\begin{align*}
E(S)=\frac{1}{q^2}\sum_{m=0}^{q+1}E(X_m)=\frac{1}{q^2}\sum_{m=0}^{q+1}\big(\binom{q+1}{m}(\frac{1}{q})^m(1-\frac{1}{q})^{q+1-m}\big)q^2=1,
\end{align*}
according to the proof of \cite[Theorem 1.]{ste}.\\
Then we determine the variance of $S$:
\begin{align*}
Var(S)&=Var(\frac{1}{q^2}\sum_{m=0}^{q+1}X_m)=E\big((\frac{1}{q^2}\sum_{m=0}^{q+1}X_m)^2\big)-E(S)^2\\
&=\frac{1}{q^4}\sum_{(i,j)\in\{0,...,q+1\}^2}\sum_{B\in\Omega}\big(\sum_{P\in{\mathbb{F}_q}^2}f_{i,P}(B)\sum_{Q\in{\mathbb{F}_q}^2}f_{j,Q}(B)\big) \p\{B\}-1^2\\
&=\frac{1}{q^4}\sum_{(i,j)\in\{0,...,q+1\}^2}\sum_{B\in\Omega}\sum_{P,Q\in{\mathbb{F}_q}^2}\big(f_{i,P}(B)f_{j,Q}(B) \p\{B\}\big)-1\\
&=\frac{1}{q^4}\sum_{(i,j)\in\{0,...,q+1\}^2}\sum_{P,Q\in{\mathbb{F}_q}^2}\p\{f_{i,P}=f_{j,Q}=1\}-1.
\end{align*}
Recall henceforth multinomial coefficients definition. Let $n$ and $p$ be positive integers. For $i\in\{1,2,...,p\}$, let $k_i$ be positive integers 
such as $\sum_{i=1}^pk_i=n$. We have:
\begin{align*}
\binom{n}{k_1,k_2,...,k_p}=\frac{n!}{k_1!k_2!...k_p!}.
\end{align*}

We note that the multinomial coefficient $\binom{a}{b,c,d}$ (where $(a,b,c,d)\in\mathbb{N^*}\times\mathbb{Z}^3$) will be considered 
to be zero, if $b$, $c$ or $d$ is strictly negative.

We could demonstrate, in the same way as in the proof of \cite[Theorem 1.]{ste}, that for two distinct points $P$ and $Q$ in ${\mathbb{F}_q}^2$ and for 
$(i,j)\in\{1,...,q\}^2$ such that $i+j\le q$:
\begin{center}
$\p\{f_{i,P}=f_{j,Q}=1\}=\binom{q}{i-1,j-1,q-i-j+2}\frac{(q-2)^{q-(i-1)-(j-1)}}{q^{q+1}}+\binom{q}{i, j, q-i-j}\frac{(q-1)(q-2)^{q-i-j}}{q^{q+1}}$.
\end{center}
The first term of the above sum corresponds to the case where $(PQ)\in B$, the second term to the case where $(PQ)\notin B$.\\
Let us consider the cases not examined yet:
\begin{enumerate}[label=$\circ$]
 \item If $i=0$ or $j=0$ (and $i+j\le q$), the above equality is still valid; indeed, for all $B\in\Omega$, $(PQ)\notin B$ (by analogy with the proof of 
\cite[Theorem 1.]{ste}), the first term of the equality is $0$.
 \item If $i+j=q+1$ (respectively $q+2$), for all $B\in\Omega$, the case where $(PQ)\notin B$ (by analogy with \cite[Theorem 1.]{ste}) cannot occur. 
 Otherwise $(PQ)$ would intersect with at least $i+j=q+1$ lines (respectively $q+2$), which is impossible in a minimal Besicovitch arrangement composed 
 of $q+1$ lines. The above equality is still valid, its second term being $0$. Moreover the particular case where $i+j=q+1$ and $i$ or $j$ is $0$ 
 follows the same rule as in the above item.
 \item If $i+j> q+2$, $(PQ)$ would intersect with at least $q+1$ lines ($(i-1)+(j-1)\ge q+1$), which is impossible as we just saw.
\end{enumerate}
Therefore the equality is always valid. Then, since $\p\{f_{i,P}=f_{j,Q}=1\}$ does not depend on $P$ and $Q$ and $\p\{f_{i,P}=1\}$ does not depend on $P$, 
we have: 
\begin{equation}\label{var2}
Var(S)=\frac{q^2-1}{q^2}\sum_{(i,j)\in\{0,...,q+1\}^2}\p\{f_{i,P}=f_{j,Q}=1\}+\frac{1}{q^2}\sum_{i\in\{0,...,q+1\}}\p\{f_{i,P}=1\}-1.
\end{equation}
The first term of (\ref{var2}) cuts into two parts. One checks that the first part is:
\begin{align*}
&(1-\frac{1}{q^2})\frac{1}{q^{q+1}}\sum_{(i,j)\in\{0,...,q+1\}^2}\binom{q}{i-1,j-1,q-i-j+2}(q-2)^{q-(i-1)-(j-1)}\\
&=(1-\frac{1}{q^2})\frac{q^{q}}{q^{q+1}}.
\end{align*}
The second part is:
\begin{align*}
(1-\frac{1}{q^2})\frac{q-1}{q^{q+1}}\sum_{(i,j)\in\{0,...,q+1\}^2}\binom{q}{i,j,q-i-j}(q-2)^{q-i-j}=(1-\frac{1}{q^2})\frac{(q-1)q^{q}}{q^{q+1}}.
\end{align*}
The second term of (\ref{var2}) is, thanks to the value of $\p\{f_{i,P}=1\}$ in the proof of \cite[Theorem 1.]{ste}:
\begin{align*}
\frac{1}{q^2}\sum_{i\in\{0,...,q+1\}}(\binom{q+1}{i}(\frac{1}{q})^i(1-\frac{1}{q})^{q+1-i})=\frac{1}{q^2}
\end{align*}
We finally obtain:
\begin{align*}
Var(S)=(1-\frac{1}{q^2})(\frac{q^{q}}{q^{q+1}}+\frac{(q-1)q^{q}}{q^{q+1}})+\frac{1}{q^2}-1=0.
\end{align*}
\end{proof}
Hence, we can deduce from this proposition that, for all $B\in\Omega$, $S(B)=1$ which can also be written as:
\begin{equation}\label{equ1}
\forall B\in\Omega,\qquad\sum_{m=0}^{q+1}x_m^B=q^2.
\end{equation}
This first trivial equality is based on the fact that $|{\mathbb{F}_q}^2|=q^2$, regardless of the arrangement chosen.

\subsubsection{The second equality}

Let us also consider the random variable $\overline{M}$ that maps $B\in\Omega$ to the mean multiplicity of the points in ${\mathbb{F}_q}^2$, associated 
with the minimal Besicovitch arrangement $B$:
\begin{align*}
\overline{M}=\frac{1}{q^2}\sum_{m=0}^{q+1}mX_m.
\end{align*}
We have the following results:
\begin{proposition}\label{Esp}
$E(\overline{M})=1+\frac{1}{q}\quad$ and $\quad Var(\overline{M})=0.$
\end{proposition}
\begin{proof}
Using the proof of \cite[Theorem 1.]{ste}, we first compute the expected value of $\overline{M}$:
\begin{align*}
E(\overline{M})&=\frac{1}{q^2}\sum_{m=0}^{q+1}mE(X_m)=\frac{1}{q^2}\sum_{m=0}^{q+1}m\big(\binom{q+1}{m}(\frac{1}{q})^m(1-\frac{1}{q})^{q+1-m}\big)q^2\\
&=\frac{q+1}{q}\sum_{m=1}^{q+1}\binom{q}{m-1}(\frac{1}{q})^{m-1}(1-\frac{1}{q})^{q-(m-1)}=1+\frac{1}{q}.
\end{align*}
Then we determine the variance of $\overline{M}$:
\begin{align*}
Var(\overline{M})&=Var(\frac{1}{q^2}\sum_{m=0}^{q+1}mX_m)=E\big((\frac{1}{q^2}\sum_{m=0}^{q+1}mX_m)^2\big)-E(\overline{M})^2\\
&=\frac{1}{q^4}\sum_{(i,j)\in\{0,...,q+1\}^2}\sum_{B\in\Omega}\big(\sum_{P\in{\mathbb{F}_q}^2}if_{i,P}(B)\sum_{Q\in{\mathbb{F}_q}^2}jf_{j,Q}(B)\big) \p\{B\}-(1+\frac{1}{q})^2\\
&=\frac{1}{q^4}\sum_{(i,j)\in\{0,...,q+1\}^2}\sum_{P,Q\in{\mathbb{F}_q}^2}ij\p\{f_{i,P}=f_{j,Q}=1\}-(1+\frac{1}{q})^2.
\end{align*}
Using the same reasoning as in Proposition \ref{Esp''}, we have:
\begin{equation}\label{var3}
\begin{split}
Var(\overline{M})=&\frac{q^2-1}{q^2}\sum_{(i,j)\in\{0,...,q+1\}^2}ij\p\{f_{i,P}=f_{j,Q}=1\}\\
&+\frac{1}{q^2}\sum_{i\in\{0,...,q+1\}}i^2\p\{f_{i,P}=1\}-(1+\frac{1}{q})^2.
\end{split}
\end{equation}
The first term of (\ref{var3}) cuts into two parts. Knowing that $ij=(i-1)(j-1)+(i-1)+(j-1)+1$, one checks that the first part is:
\begin{align*}
&(1-\frac{1}{q^2})\sum_{(i,j)\in\{0,...,q+1\}^2}ij\binom{q}{i-1,j-1,q-i-j+2}\frac{(q-2)^{q-(i-1)-(j-1)}}{q^{q+1}}\\
&=(1-\frac{1}{q^2})\frac{1}{q^{q+1}}\big(q(q-1)q^{q-2}+qq^{q-1}+qq^{q-1}+q^q\big)\\
&=(1-\frac{1}{q^2})(\frac{4q-1}{q^2})
\end{align*}
The second part is:
\begin{align*}
&(1-\frac{1}{q^2})\sum_{(i,j)\in\{0,...,q+1\}^2}ij\binom{q}{i, j, q-i-j}\frac{(q-1)(q-2)^{q-i-j}}{q^{q+1}}\\
&=(1-\frac{1}{q^2})\frac{1}{q^{q+1}}\big((q-1)q(q-1)q^{q-2}\big)
\end{align*}
The second term of (\ref{var3}) is, thanks to the value of $\p\{f_{i,P}=1\}$ in the proof of \cite[Theorem 1.]{ste} (knowing that $i^2=i(i-1)+i$):
\begin{align*}
\frac{1}{q^2}\sum_{i\in\{0,...,q+1\}}\big(i(i-1)+i\big)\binom{q+1}{i}(\frac{1}{q})^i(1-\frac{1}{q})^{q+1-i}&=\frac{1}{q^2}\big(\frac{(q+1)q}{q^2}+\frac{q+1}{q}\big)\\
&=\frac{2(q+1)}{q^3}.
\end{align*}
Adding the different terms of (\ref{var3}), we get:
\begin{align*}
Var(\overline{M})=(1-\frac{1}{q^2})(\frac{4q-1}{q^2}+\frac{(q-1)^2}{q^2})+(\frac{2(q+1)}{q^3})-1-\frac{1}{q^2}-\frac{2}{q}=0.
\end{align*}
\end{proof}
Hence, we can deduce from this proposition that, for all $B\in\Omega$, $\overline{M}(B)=1+\frac{1}{q},$ which can also be written as:
\begin{equation}\label{equ2}
\forall B\in\Omega,\qquad\sum_{m=0}^{q+1}mx_m^B=q(q+1).
\end{equation}
This result can simply be interpreted as follows: the sum of the multiplicity of all points in ${\mathbb{F}_q}^2$ is $q(q+1)$, regardless of the arrangement of 
$q+1$ lines chosen. Indeed, there are $q$ points on each of the $q+1$ lines, and therefore, there are $q(q+1)$ points counted with multiplicity.

\begin{remark}
Let us denote by $\mathscr{P}$ the subset of $\mathbb{N}$ constituted of all prime powers. Let us consider, for $P\in{\mathbb{F}_q}^2$, the random variable 
$M_P$ defined in the beginning of the section. Recall that $M_P$ follows a binomial distribution with parameters $q+1$ and $1/q$ (see the proof of 
\cite[Theorem 1.]{ste}), its expected value being $1+\frac{1}{q}$ and its standard deviation being $\sqrt{\frac{(q+1)(q-1)}{q^2}}$. 
$\{M_P:P\in{\mathbb{F}_q}^2\}$ is a set of identically distributed random variables (indeed, the value of $\p\{f_{i,P}=1\}$ doesnot depend on $P$). Let us denote 
by $(S_q)_{q\in\mathscr{P}}$ the family of random variables where, for $q\in\mathscr{P}$:
\begin{align*}
S_q=\frac{(\sum_{P\in{\mathbb{F}_q}^2}M_P)-q^2(1+\frac{1}{q})}{q\sqrt{\frac{(q+1)(q-1)}{q^2}}}=\frac{(\sum_{P\in{\mathbb{F}_q}^2}M_P)-q(q+1)}{\sqrt{q^2-1}}.
\end{align*}
If $\{M_P:P\in{\mathbb{F}_q}^2\}$ is a set of independent random variables, the central limit theorem will state that the indexed family 
$(S_q)_{q\in\mathscr{P}}$ converges in law towards the standard normal distribution.\\
However, Proposition \ref{Esp} shows us that, for all $q\in\mathscr{P}$, $S_q$ is the constant random variable with value $O$. Thus, the random 
variables of $\{M_P:P\in{\mathbb{F}_q}^2\}$ are dependent.
\end{remark}

\subsubsection{The third equality}

Let us eventually consider the random variable $V$ that maps $B\in\Omega$ to the second moment of the multiplicity of the points in ${\mathbb{F}_q}^2$, associated 
with the minimal Besicovitch arrangement $B$:
\begin{align*}
V=\frac{1}{q^2}\sum_{m=0}^{q+1}m^2X_m.
\end{align*}
We have the following results:
\begin{proposition}\label{Esp'}
$E(V)=2+\frac{2}{q}\quad$ and $\quad Var(V)=0.$
\end{proposition}
\begin{proof}
We first compute the expected value of $V$, using the proof of \cite[Theorem 1.]{ste}:
\begin{align*}
E(V)=\frac{1}{q^2}\sum_{m=0}^{q+1}m^2E(X_m)=\frac{1}{q^2}\sum_{m=1}^{q+1}m^2\big(\binom{q+1}{m}(\frac{1}{q})^m(1-\frac{1}{q})^{q+1-m}\big)q^2=2+\frac{2}{q},
\end{align*}
the last equality being given by the calculation of the second term of (\ref{var3}) in the proof of Proposition \ref{Esp}.\\
Then we determine the variance of $V$:
\begin{align*}
Var(V)&=Var(\frac{1}{q^2}\sum_{m=0}^{q+1}m^2X_m)=E\big((\frac{1}{q^2}\sum_{m=0}^{q+1}m^2X_m)^2\big)-E(V)^2\\
&=\frac{1}{q^4}\sum_{(i,j)\in\{0,...,q+1\}^2}\sum_{B\in\Omega}\big(\sum_{P\in{\mathbb{F}_q}^2}i^2f_{i,P}(B)\sum_{Q\in{\mathbb{F}_q}^2}j^2f_{j,Q}(B)\big)
\p\{B\}\\
&-(2+\frac{2}{q})^2\\
&=\frac{1}{q^4}\sum_{(i,j)\in\{0,...,q+1\}^2}\sum_{P,Q\in{\mathbb{F}_q}^2}i^2j^2\p\{f_{i,P}=f_{j,Q}=1\}-(2+\frac{2}{q})^2.
\end{align*}
Using the same reasoning as in Proposition \ref{Esp''}, we have:
\begin{equation}\label{var4}
\begin{split}
Var(V)&=\frac{q^2-1}{q^2}\sum_{(i,j)\in\{0,...,q+1\}^2}i^2j^2\p\{f_{i,P}=f_{j,Q}=1\}\\
&+\frac{1}{q^2}\sum_{i\in\{0,...,q+1\}}i^4\p\{f_{i,P}=1\}-(2+\frac{2}{q})^2.
\end{split}
\end{equation}
The first term of (\ref{var4}) cuts into two parts. Knowing that:
\begin{center}
$i^2j^2=i_1i_2j_1j_2+3i_1i_2j_1+3i_1j_1j_2+9i_1j_1+i_1i_2+j_1j_2+3i_1+3j_1+1$, 
\end{center}
where $i_1=i-1$, $i_2=i-2$, $j_1=j-1$ and $j_2=j-2$, one checks that the first part is:
\begin{align*}
&(1-\frac{1}{q^2})\sum_{(i,j)\in\{0,...,q+1\}^2}i^2j^2\binom{q}{i-1,j-1,q-i-j+2}\frac{(q-2)^{q-(i-1)-(j-1)}}{q^{q+1}}\\
&=(1-\frac{1}{q^2})\big(\frac{q(q-1)(q-2)(q-3)}{q^5}+\frac{6q(q-1)(q-2)}{q^4}+\frac{11q(q-1)}{q^3}+\frac{6q}{q^2}+\frac{1}{q}\big)\\
&=\frac{25q^5-35q^4-2q^3+29q^2-23q+6}{q^6}.
\end{align*}
Knowing that $i^2j^2=i(i-1)j(j-1)+i(i-1)j+ij(j-1)+ij$, the second part is:
\begin{align*}
&(1-\frac{1}{q^2})\sum_{(i,j)\in\{0,...,q+1\}^2}i^2j^2\binom{q}{i, j, q-i-j}\frac{(q-1)(q-2)^{q-i-j}}{q^{q+1}}\\
&=(1-\frac{1}{q^2})\big(\frac{q(q-1)^2(q-2)(q-3)}{q^5}+\frac{2q(q-1)^2(q-2)}{q^4}+\frac{q(q-1)^2}{q^3}\big)\\
&=\frac{4q^6-17q^5+24q^4-4q^3-22q^2+21q-6}{q^6}.
\end{align*}
The second term of (\ref{var4}) is, thanks to the value of $\p\{f_{i,P}=1\}$ in the proof of \cite[Theorem 1.]{ste} (and with the help of the 
following equality $i^4=i(i-1)(i-2)(i-3)+6i(i-1)(i-2)+7i(i-1)+i$):
\begin{align*}
&\frac{1}{q^2}\sum_{i\in\{1,...,q+1\}}i^4\binom{q+1}{m}(\frac{1}{q})^m(1-\frac{1}{q})^{q+1-m}\\
&=\frac{1}{q^2}\big(\frac{(q+1)q(q-1)(q-2)}{q^4}+\frac{6(q+1)q(q-1)}{q^3}+\frac{7(q+1)q}{q^2}+\frac{q+1}{q}\big)\\
&=\frac{15q^3+6q^2-7q+2}{q^5}.
\end{align*}
Adding the different terms of (\ref{var4}), we obtain the expected result.
\end{proof}
We can deduce again from this proposition that, for all $B\in\Omega$, $V(B)=2+\frac{2}{q},$ which can also be written as:
\begin{equation}\label{equ3}
\forall B\in\Omega,\qquad\sum_{m=0}^{q+1}m^2x_m^B=2q(q+1).
\end{equation}

\begin{remark}
We can recover the Incidence Formula in \cite{Faber}, $|\tilde{B}|=\frac{q(q+1)}{2}+\sum_{m=1}^{q+1}\frac{(m-1)(m-2)}{2}x_m^B$, from equalities 
(\ref{equ2}) and (\ref{equ3}). Indeed, for $B$ a minimal Besicovitch arrangement we have:
\[|\tilde{B}|=\sum_{m=1}^{q+1}x_m^B=x_1^B+x_2^B+\sum_{m=3}^{q+1}x_m^B.\] 
We deduce from equalities (\ref{equ2}) and (\ref{equ3}): 
\begin{equation}\label{equ4}
x_1^B=\sum_{m=3}^{q+1}(m^2-2m)x_m^B\quad\text{and}\quad x_2^B=\frac{q(q+1)}{2}+\sum_{m=3}^{q+1}\frac{-m^2+m}{2}x_m^B.
\end{equation}
Then $|\tilde{B}|$ is equal to $\frac{q(q+1)}{2}+\sum_{m=3}^{q+1}\frac{m^2-3m+2}{2}x_m^B$.
\end{remark}

We note that the second equality in (\ref{equ4}) can also be found as follows. If all the intersection points between two lines of our $q+1$ lines 
arrangement are distinct, there will be $\binom{q+1}{2}$ points of multiplicity $2$. Let $i$ be an integer such that $3\le i\le q+1$. A point of 
multiplicity $i$ is a point through which $i$ lines pass. It implies that the number of points of multiplicity $2$ is reduced by $\binom{i}{2}$ for each of 
these points. Hence:
\begin{align*}
x_2^B=\binom{q+1}{2}-\sum_{i=3}^{q+1}\binom{i}{2}x_i^B.
\end{align*}

\subsubsection{Topological invariants in minimal Besicovitch arrangements in \texorpdfstring{${\mathbb{F}_q}^2$}{} and examples}

The three equalities ((\ref{equ1}), (\ref{equ2}) and (\ref{equ3})) associated with minimal Besicovitch arrangements show some of the constraints they 
are subject to:
\begin{corollary}\label{sys}
For $B\in\Omega$ : 
\[\left \{\begin{array}{l @{\enspace=\enspace} l}
    \sum_{m=0}^{q+1}x_m^B & q^2 \\
    \sum_{m=0}^{q+1}mx_m^B & q(q+1)\\
    \sum_{m=0}^{q+1}m^2x_m^B & 2q(q+1).
\end{array}
\right.\]
\end{corollary}

\begin{remark}
It may be noted that the values of $\sum_{m=0}^{q+1}m^3x_m^B$ are far from being always the same for $B$ in ${\mathbb{F}_q}^2$. 
\end{remark}

Let us point out the existence of two particular examples which will be useful in the remainder of this section. Here is the first one: 
\begin{example}\label{ex1}
It is the extreme case where all the lines of the minimal Besicovitch arrangement $B_0$ are concurrent: $B_0=\bigcup_{i\in{\mathbb{F}_q}\cup\{\infty\}}l(i,0)$ for example. 
We have $x_1^{B_0}=q^2-1$, $x_{q+1}^{B_0}=1$ and $x_i^{B_0}=0$ if $i\neq1,q+1$.
\end{example}
Faber produces the second one in \cite[Example.]{Faber}: 
\begin{example}\label{ex2}
It is the minimal Besicovitch arrangement defined as below:\\
$B_1=\big(\bigcup_{i\in{\mathbb{F}_q}}l(i,-i^2)\big)\cup l(\infty,0)$. \\
If $q$ is odd, $x_0^{B_1}=\frac{(q-1)^2}{2}$, $x_1^{B_1}=\frac{3q-3}{2}$, $x_2^{B_1}=\frac{q^2-2q+3}{2}$, $x_3^{B_1}=\frac{q-1}{2}$ and $x_i^{B_1}=0$ if 
$i\neq0,1,2,3$. \\
If $q$ is even, $x_0^{B_1}=\frac{q(q-1)}{2}$, $x_1^{B_1}=0$, $x_2^{B_1}=\frac{q(q+1)}{2}$ and $x_i^{B_1}=0$ if $i\neq0,1,2$.
\end{example}

\subsection{Resulting inequalities}

The following proposition can be deduced from the foregoing:
\begin{proposition}
Let $m$ and $q$ be positive integers such that $m\le q+1$. We have: 
\begin{align*}
&max(\{x_m^B:B\in\Omega\})\le\frac{2q(q+1)}{m^2}\quad \text{for } m\ge2,\\
&max(\{x_0^B:B\in\Omega\})\le\frac{q(q-1)}{2} \quad\text{and}\quad max(\{x_1^B:B\in\Omega\})=q^2-1.
\end{align*}
Furthermore, as $q\rightarrow\infty$:
\begin{center}
$max(\{x_0^B:B\in\Omega\})\backsim 1/2q^2$.
\end{center}
\end{proposition}
\begin{proof}
Let $B\in\Omega$ then we obtain thanks to equalities (\ref{equ1}) and (\ref{equ2}):
\[x_0^B+q=\sum_{m=2}^{q+1}(m-1)x_m^B.\]\\
Thanks to equalities (\ref{equ2}) and (\ref{equ3}), we get as well: 
\[q(q+1)=\sum_{m=2}^{q+1}m(m-1)x_m^B\ge2\sum_{m=2}^{q+1}(m-1)x_m^B=2(x_0^B+q).\]\\
Thus:
\begin{center}
$x_0^B\le\frac{q(q-1)}{2}.$
\end{center}
Example \ref{ex2} allows us to conclude that $max(\{x_0^B:B\in\Omega\})\backsim 1/2q^2$, as $q$ goes to infinity.

In addition, Example \ref{ex1} shows us that $max(\{x_1^B:B\in\Omega\})\ge q^2-1$. But the case where the number of simple points is equal to $q^2$ doesn't 
occur, because there is at least one point of multiplicity greater or equal to $2$ in a minimal Besicovitch arrangement. 

Eventually, for $B\in\Omega$ and for $2\le m\le q+1$, equality (\ref{equ3}) gives us: 
\begin{center}
$x_m^B\le\frac{2q(q+1)}{m^2}.$ 
\end{center}
\end{proof}
More precisely, we obtain the following two results. If $q$ is even, $max(\{x_0^B:B\in\Omega\})=q(q-1)/2$ (Example \ref{ex2}) and if $q$ is odd, $max
(\{x_0^B:B\in\Omega\})=(q-1)^2/2$ (see \cite[Proposition 7]{Blokhuis} and also Example \ref{ex2}).

For $B\in\Omega$, by considering only $x_0^B$, $x_1^B$ and $x_2^B$, we highlight three inequalities:
\begin{proposition}\label{ze}
For $B\in\Omega$:
\[\left \{\begin{array}{r @{\enspace} l}
    x_0^B+x_1^B+x_2^B &\le\enspace q^2 \\
    3x_0^B-x_2^B &\le\enspace q^2-2q \\
    3x_0^B+2x_1^B+x_2^B &\ge\enspace 2q^2-q.
\end{array}
\right.\]
\end{proposition}
\begin{proof}
The first inequality derives directly from equality (\ref{equ1}).
The remaining two rely on equalities (\ref{equ4}): 
\begin{center}
$x_1^B=\sum_{m=3}^{q+1}(m^2-2m)x_m^B$ and $x_2^B=\frac{q(q+1)}{2}+\sum_{m=3}^{q+1}\frac{-m^2+m}{2}x_m^B$.
\end{center}
Indeed:
\begin{align*}
3x_0^B-x_2^B&=3(q^2-\sum_{m=1}^{q+1}x_m^B)-x_2^B=3q^2-3x_1^B-4x_2^B-3\sum_{m=3}^{q+1}x_m^B\\
&=q^2-2q+\sum_{m=3}^{q+1}(-m^2+4m-3)x_m^B\enspace\le\enspace q^2-2q\quad\text{as }m\ge3,
\end{align*}
\begin{align*}
3x_0^B+2x_1^B+x_2^B&=3(q^2-\sum_{m=1}^{q+1}x_m^B)+2x_1^B+x_2^B=3q^2-x_1^B-2x_2^B-3\sum_{m=3}^{q+1}x_m^B\\
&=2q^2-q+\sum_{m=3}^{q+1}(m-3)x_m^B\enspace\ge\enspace 2q^2-q\quad\text{as }m\ge3.
\end{align*}
\end{proof}
%In Example \ref{ex1}, we have $x_0^{B_0}+x_1^{B_0}+x_2^{B_0}=q^2-1$ and $3x_0^{B_0}+2x_1^{B_0}+x_2^{B_0}=2q^2-2$.
Regarding Example \ref{ex2}, the sharpness of the three inequalities of Proposition \ref{ze} can be emphasised: if $q$ is even, all of these three 
inequalities are sharp. 
%If $q$ is odd, we have $x_0^{B_1}+x_1^{B_1}+x_2^{B_1}=q^2-\frac{q-1}{2}$ whereas the second and the third inequalities become equalities. 

Let us recall that, for $B\in\Omega$, $0\le x_i^B/q^2\le 1$ if $i=0,1,2$ as $|{\mathbb{F}_q}^2|=q^2$. Let us then place ourselves in the unit cube of 
$\mathbb{R}^3$ and consider the subset of points $(x,y,z)$ of $\mathbb{R}^3$, displayed in green in the figure below, defined by: 
\[\left \{\begin{array}{r @{\enspace} l}
    x+y+z &\le\enspace 1 \\
    3x-z &\le\enspace 1 \\
    3x+2y+z &\ge\enspace 2.
\end{array}
\right.\]
As, for a sufficiently large $q$, the terms $q$ or $2q$ of the inequalities of Proposition \ref{ze} can be neglected compared to $q^2$, all 
$(x_0^B/q^2,x_1^B/q^2,x_2^B/q^2)$, for $B\in\Omega$, are in the green subset with a $O(1/q)$ accuracy, as $q\rightarrow\infty$.

\begin{figure}[htbp]
\centering
\includegraphics[scale=0.7,viewport=150 120 500 400,clip=true]{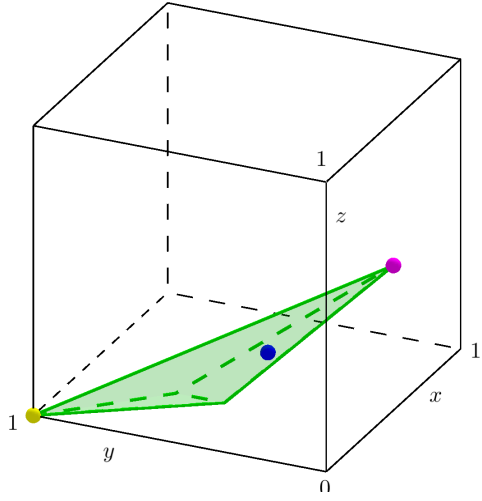}
\caption{The subset of the unit cube of $\mathbb{R}^3$ containing the points associated with the minimal Besicovitch arrangements.}
\end{figure}

Most minimal Besicovitch arrangements would be represented by points located near the blue point with coordinates $\big(1/e,1/e,1/(2e)\big)$, according 
to \cite[Theorem 1.]{ste}. The point associated with Example \ref{ex2} and that one associated with Example \ref{ex1} tend respectively to the 
mauve and the yellow point, as $q$ goes to infinity.

\section{Connections between \texorpdfstring{${\mathbb{F}_q}^2$}{} and \texorpdfstring{$\mathbb{R}^2$}{}}

In this section we will consider the definition of a finite incidence structure given in \cite{Pott} in a affine plane:
\begin{definition}
A finite incidence structure is a triple $(\mathcal{P},\mathcal{L},\mathcal{I})$, where $\mathcal{P}$ and $\mathcal{L}$ are two finite disjoint sets and 
$\mathcal{I}$ is a subset such that $\mathcal{I}\subseteq\mathcal{L}\times\mathcal{P}$. Elements in $\mathcal{P}$ are the points of the affine plane, 
elements in $\mathcal{L}$ are the lines, whereas $\mathcal{I}$ is the incidence relation.
\end{definition}

For $P\in\mathcal{P}$ and $L\in\mathcal{L}$, $(L,P)\in\mathcal{I}$ if and only if $P\in L$.

For $A$ an arrangement in an affine plane, we will use the following notation:
\begin{enumerate}[label=$\bullet$]
 \item $p_1(A)$ is the cardinality of $A$;
 \item $p_{01}(A)$ is the number of pairs consisting of a line of $A$ and one of its point of intersection;
 \item $x_i(A)$ is the number of points of multiplicity $i$ associated with $A$, for $i\ge2$;
 \item $p_0(A)$ is the number of points defined by $A$ (without counting multiplicity). 
\end{enumerate}
Note that $p_0(A)=\sum_{i=2}^{|A|}x_i(A)$.

Henceforth, let us define the following equivalence relation:
\begin{definition}
Let $A$ and $B$ be two arrangements in the plane consisting of at least two lines.\\
$(\mathcal{P}_A,\mathcal{L}_A,\mathcal{I}_A)$ is the incidence structure of $A$ and $(\mathcal{P}_B,\mathcal{L}_B,\mathcal{I}_B)$ is that of $B$. The 
equivalence relation $\mathcal{R}$ is defined by:
\[(\mathcal{P}_A,\mathcal{L}_A,\mathcal{I}_A)\mathcal{R}(\mathcal{P}_B,\mathcal{L}_B,\mathcal{I}_B)\enspace\Longleftrightarrow\enspace|A|=|B|\enspace
\text{and}\enspace\forall i\in\{2,...,|A|\},\text{ }x_i(A)=x_i(B).\]
\end{definition}
The equivalence class of $(\mathcal{P}_A,\mathcal{L}_A,\mathcal{I}_A)$ will be denoted by $[x_2(A),...,x_{|A|}(A)]$.

This being so, let us study the connections between ${\mathbb{F}_q}^2$ and $\mathbb{R}^2$, where $q$ is a prime power and ${\mathbb{F}_q}$ the $q$ elements 
finite field.

Let us denote by $\mathcal{S}_{{\mathbb{F}_q}^2}$ the set of incidence structures in ${\mathbb{F}_q}^2$ and $\mathcal{S}_{\mathbb{R}^2}$ the set of 
incidence structures in $\mathbb{R}^2$. Let $\Phi$ be the identity function from $\mathcal{S}_{{\mathbb{F}_q}^2}/\mathcal{R}$ to $\mathcal{S}_{\mathbb{R}^2}/
\mathcal{R}$ mapping an element of $\mathcal{S}_{{\mathbb{F}_q}^2}/\mathcal{R}$ to that one of $\mathcal{S}_{\mathbb{R}^2}/\mathcal{R}$ with the same 
notation.

Our aim is here to identify some connections between an arrangement $F$ of lines in ${\mathbb{F}_q}^2$ and an arrangement $A$ in $\mathbb{R}^2$, such 
that $A\in\Phi([x_2(F),...,x_{|F|}(F)])$.

To this end, we define quantities specific to ${\mathbb{F}_q}^2$ and $\mathbb{R}^2$. 

For $F$ an arrangement of lines in ${\mathbb{F}_q}^2$:
\begin{enumerate}[label=$\bullet$]
 \item $x_i(F)$ is the number of points of multiplicity $i$ associated with $F$, for $i\in\{0,1\}$.
\end{enumerate}

For $A$ an arrangement of lines in $\mathbb{R}^2$ (see \cite{Cartier}):
\begin{enumerate}[label=$\bullet$]
 \item $f_i(A)$ is the number of $i$-dimensional cells determined by $A$, for $i\in\{0,1,2\}$, where $0$-dimensional cells are the vertices of $A$, 
$1$-dimensional cells the edges of $A$ and $2$-dimensional cells the faces of $A$;
 \item $f_i^b(A)$ is the number of bounded $i$-dimensional cells for $i\in\{1,2\}$.
\end{enumerate}

There exist four relationships between the different mathematical quantities in $\mathbb{R}^2$ (see \cite{Cartier}), listed below:

\begin{proposition}\label{form}
For all arrangements $A$ of lines in $\mathbb{R}^2$: 
\begin{align*}
&f_1(A)=p_1(A)+p_{01}(A),\quad f_2(A)=1-p_0(A)+p_1(A)+p_{01}(A),\\
&f_1^b(A)=-p_1(A)+p_{01}(A),\quad f_2^b(A)=1-p_0(A)-p_1(A)+p_{01}(A).
\end{align*}
\end{proposition}

Let $F$ be an arrangement of $q+1$ lines in ${\mathbb{F}_q}^2$. Assuming henceforth the existence of an arrangement $A$ in $\mathbb{R}^2$, such that 
$A\in\Phi([x_2(F),...,x_{|F|}(F)])$.

One has $p_1(A)=p_1(F)=q+1$ (be $A\in\Phi([x_2(F),...,x_{q+1}(F)])$). We also have $p_0(F)=\sum_{i=2}^{q+1}x_i(F)=p_0(A)$, since for all $i\in\{2,...,q+1\}$ 
$x_i(A)=x_i(F)$. For a fixed point $P$ in $\mathbb{R}^2$ or in ${\mathbb{F}_q}^2$, the number of pairs consisting of a line of $A$ passing through $P$ 
and the point $P$ itself is by definition the multiplicity of this point. Therefore $p_{01}(A)$ and $p_{01}(F)$ are equal to the sum of the multiplicity 
of all points of intersection: $p_{01}(F)=\sum_{i=2}^{q+1}ix_i(F)=p_{01}(A)$. 

In this context, we can outline the following theorem:

\begin{theorem}\label{lienFR}
Let $F$ be an arrangement of $q+1$ lines in ${\mathbb{F}_q}^2$.\\
Let us assume that there exists an arrangement $A$ in $\mathbb{R}^2$, such that:
\begin{center}
$A\in\Phi([x_2(F),...,x_{|F|}(F)])$.
\end{center}
Then:
\begin{align*}
f_1^b(A)=(q+1)(q-1)-x_1(F)\enspace\text{and}\enspace f_2^b(A)=x_0(F).
\end{align*}
\end{theorem}

\begin{proof}
According to Proposition \ref{form}, we know: $f_1^b(A)=-p_1(A)+p_{01}(A)=-(q+1)+\sum_{i=2}^{q+1}ix_i(F)$. Equality (\ref{equ2}) gives us: 
$f_1^b(A)=-(q+1)+q(q+1)-x_1(F)$ \textit{i.e.} the first equality of the theorem.

For the second equality, according to Proposition \ref{form} we have: $f_2^b(A)=1-p_0(A)-p_1(A)+p_{01}(A)=1-\sum_{i=2}^{q+1}x_i(F)-(q+1)+
\sum_{i=2}^{q+1}ix_i(F)$. From the equalities $\sum_{i=0}^{q+1}x_i(F)=q^2$ and $\sum_{i=0}^{q+1}ix_i(F)=q(q+1)$ (which are true for 
all sets of $q+1$ lines), we deduce that $f_2^b(A)=1-(q^2-x_0(F)-x_1(F))-(q+1)+(q(q+1)-x_1(F))=x_0(F)$.
\end{proof}

The first equality of Theorem \ref{lienFR} can be explained as follows. When we add a line to an arrangement, two cases appear. The first one is the 
case where each new point (resulting from the intersection of the new line and one of the others) leads at the same time to the appearance of a 
bounded one-dimensional cells for each of the two lines and the disappearance of a single point on each of these lines. The second one is the case where 
the new line passes through an already existing point; here, only one new bounded $1$-dimensional cell appears and one single point disappears 
(on the new line). This explains the coefficient $-1$ of $x_1(F)$ in the first equality of the theorem. In other words, $x_1(F)+f_1^b(A)$ is a constant 
for a given $q$. In the already mentioned extreme case of concurrence of all lines of a minimal Besicovitch arrangement $B$ (Example \ref{ex1}), we 
have: $x_1(F)=(q+1)\times(q-1)$ and $f_1^b(A)=0$. Thus we recover the first equality of Theorem \ref{lienFR}.

\begin{example}
We give an example where $q=5$. $\mathbb{Z}/5\mathbb{Z}$ being a field, we can use the following notations. To obtain a minimal Besicovitch arrangement 
$F$, we choose $6$ lines in $\mathbb{F}_5^2$ with distinct slopes. Their equations are: $y=0$, $y=x+1$, $y=2x+1$, $y=3x+2$, $4x+2$ and $x=0$.
The figure below shows us the different lines of $F$ in $\mathbb{F}_5^2$ (on the left) and one of its associated arrangement $A$ in $\mathbb{R}^2$, 
verifying $A\in\Phi([x_2(F),...,x_{|F|}(F)])$ (on the right):

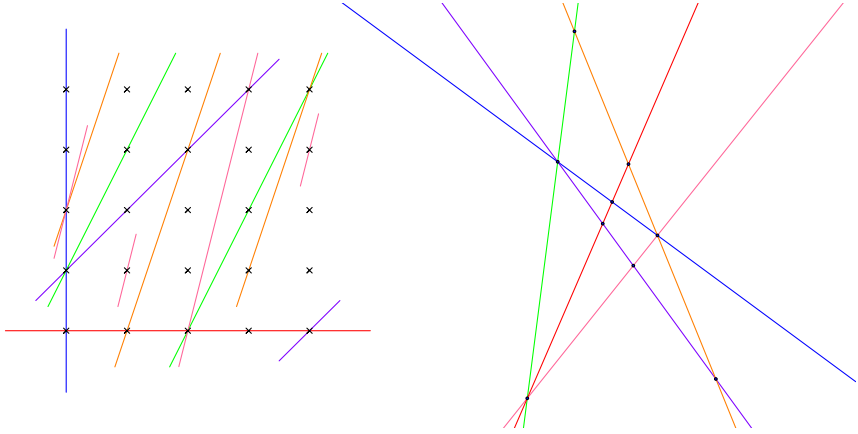
\begin{figure}[htbp]
\centering
\definecolor{ffqqqq}{rgb}{1.,0.,0.}
\definecolor{ffwwzz}{rgb}{1.,0.4,0.6}
\definecolor{ffxfqq}{rgb}{1.,0.4980392156862745,0.}
\definecolor{qqffqq}{rgb}{0.,1.,0.}
\definecolor{xfqqff}{rgb}{0.4980392156862745,0.,1.}
\definecolor{qqqqff}{rgb}{0.,0.,1.}
\begin{tikzpicture}[line cap=round,line join=round,>=triangle 45,x=0.8cm,y=0.8cm]
\clip(-2.42,-2.6) rectangle (11.98,4.42);\draw [color=xfqqff,domain=-2.42:11.98] plot(\x,{(--30.168-3.6*\x)/2.6});
\draw [color=qqffqq,domain=-2.42:11.98] plot(\x,{(--26.8536-3.92*\x)/-0.5});\draw [color=qqqqff,domain=-2.42:11.98] plot(\x,{(--11.5896-1.22*\x)/1.64});
\draw [color=ffwwzz,domain=-2.42:11.98] plot(\x,{(--22.3028-2.7*\x)/-2.14});\draw [color=ffqqqq,domain=-2.42:11.98] plot(\x,{(-20.968--2.8*\x)/1.2});
\draw [color=qqqqff] (-1.,-2.02)-- (-1.,4.);\draw [color=ffxfqq,domain=-2.42:11.98] plot(\x,{(--21.3104-2.38*\x)/0.96});
\draw [color=ffqqqq] (-2.,-1.)-- (4.,-1.);\draw [color=xfqqff] (-1.5,-0.5)-- (2.5,3.5);\draw [color=xfqqff] (3.5,-0.5)-- (2.5,-1.5);
\draw [color=qqffqq] (-1.3,-0.6)-- (0.8,3.6);\draw [color=qqffqq] (0.7,-1.6)-- (3.3,3.6);\draw [color=ffxfqq] (-1.2,0.4)-- (-0.133,3.601);
\draw [color=ffxfqq] (-0.2,-1.6)-- (1.533,3.599);\draw [color=ffxfqq] (3.2,3.6)-- (1.8,-0.6);\draw [color=ffwwzz] (0.15,0.6)-- (-0.15,-0.6);
\draw [color=ffwwzz] (0.8511764705882353,-1.595294117647059)-- (2.15,3.6);\draw [color=ffwwzz] (-0.65,2.4)-- (-1.2,0.2);
\draw [color=ffwwzz] (2.8494117647058825,1.3976470588235292)-- (3.148823529411765,2.5952941176470588);
\begin{scriptsize}
\draw [color=black] (-1.,3.)-- ++(-1.0pt,-1.0pt) -- ++(2.0pt,2.0pt) ++(-2.0pt,0) -- ++(2.0pt,-2.0pt);
\draw [color=black] (-1.,2.)-- ++(-1.0pt,-1.0pt) -- ++(2.0pt,2.0pt) ++(-2.0pt,0) -- ++(2.0pt,-2.0pt);
\draw [color=black] (0.,3.)-- ++(-1.0pt,-1.0pt) -- ++(2.0pt,2.0pt) ++(-2.0pt,0) -- ++(2.0pt,-2.0pt);
\draw [color=black] (0.,2.)-- ++(-1.0pt,-1.0pt) -- ++(2.0pt,2.0pt) ++(-2.0pt,0) -- ++(2.0pt,-2.0pt);
\draw [color=black] (1.,3.)-- ++(-1.0pt,-1.0pt) -- ++(2.0pt,2.0pt) ++(-2.0pt,0) -- ++(2.0pt,-2.0pt);
\draw [color=black] (2.,3.)-- ++(-1.0pt,-1.0pt) -- ++(2.0pt,2.0pt) ++(-2.0pt,0) -- ++(2.0pt,-2.0pt);
\draw [color=black] (3.,3.)-- ++(-1.0pt,-1.0pt) -- ++(2.0pt,2.0pt) ++(-2.0pt,0) -- ++(2.0pt,-2.0pt);
\draw [color=black] (-1.,1.)-- ++(-1.0pt,-1.0pt) -- ++(2.0pt,2.0pt) ++(-2.0pt,0) -- ++(2.0pt,-2.0pt);
\draw [color=black] (-1.,0.)-- ++(-1.0pt,-1.0pt) -- ++(2.0pt,2.0pt) ++(-2.0pt,0) -- ++(2.0pt,-2.0pt);
\draw [color=black] (-1.,-1.)-- ++(-1.0pt,-1.0pt) -- ++(2.0pt,2.0pt) ++(-2.0pt,0) -- ++(2.0pt,-2.0pt);
\draw [color=black] (0.,1.)-- ++(-1.0pt,-1.0pt) -- ++(2.0pt,2.0pt) ++(-2.0pt,0) -- ++(2.0pt,-2.0pt);
\draw [color=black] (0.,0.)-- ++(-1.0pt,-1.0pt) -- ++(2.0pt,2.0pt) ++(-2.0pt,0) -- ++(2.0pt,-2.0pt);
\draw [color=black] (0.,-1.)-- ++(-1.0pt,-1.0pt) -- ++(2.0pt,2.0pt) ++(-2.0pt,0) -- ++(2.0pt,-2.0pt);
\draw [color=black] (1.,2.)-- ++(-1.0pt,-1.0pt) -- ++(2.0pt,2.0pt) ++(-2.0pt,0) -- ++(2.0pt,-2.0pt);
\draw [color=black] (1.,1.)-- ++(-1.0pt,-1.0pt) -- ++(2.0pt,2.0pt) ++(-2.0pt,0) -- ++(2.0pt,-2.0pt);
\draw [color=black] (1.,0.)-- ++(-1.0pt,-1.0pt) -- ++(2.0pt,2.0pt) ++(-2.0pt,0) -- ++(2.0pt,-2.0pt);
\draw [color=black] (2.,0.)-- ++(-1.0pt,-1.0pt) -- ++(2.0pt,2.0pt) ++(-2.0pt,0) -- ++(2.0pt,-2.0pt);
\draw [color=black] (1.,-1.)-- ++(-1.0pt,-1.0pt) -- ++(2.0pt,2.0pt) ++(-2.0pt,0) -- ++(2.0pt,-2.0pt);
\draw [color=black] (2.,-1.)-- ++(-1.0pt,-1.0pt) -- ++(2.0pt,2.0pt) ++(-2.0pt,0) -- ++(2.0pt,-2.0pt);
\draw [color=black] (2.,1.)-- ++(-1.0pt,-1.0pt) -- ++(2.0pt,2.0pt) ++(-2.0pt,0) -- ++(2.0pt,-2.0pt);
\draw [color=black] (2.,2.)-- ++(-1.0pt,-1.0pt) -- ++(2.0pt,2.0pt) ++(-2.0pt,0) -- ++(2.0pt,-2.0pt);
\draw [color=black] (3.,2.)-- ++(-1.0pt,-1.0pt) -- ++(2.0pt,2.0pt) ++(-2.0pt,0) -- ++(2.0pt,-2.0pt);\draw [fill=qqqqff] (9.68,-1.8) circle (0.5pt);
\draw [color=black] (3.,1.)-- ++(-1.0pt,-1.0pt) -- ++(2.0pt,2.0pt) ++(-2.0pt,0) -- ++(2.0pt,-2.0pt);\draw [fill=qqqqff] (8.72,0.58) circle (0.5pt);
\draw [color=black] (3.,0.)-- ++(-1.0pt,-1.0pt) -- ++(2.0pt,2.0pt) ++(-2.0pt,0) -- ++(2.0pt,-2.0pt);\draw [fill=qqqqff] (6.58,-2.12) circle (0.5pt);
\draw [color=black] (3.,-1.)-- ++(-1.0pt,-1.0pt) -- ++(2.0pt,2.0pt) ++(-2.0pt,0) -- ++(2.0pt,-2.0pt);\draw [fill=qqqqff] (7.08,1.8) circle (0.5pt);
\draw [fill=qqqqff] (8.322928552023907,0.07902200488997495) circle (0.5pt);\draw [fill=qqqqff] (7.974742404227212,1.134398943196829) circle (0.5pt);
\draw [fill=qqqqff] (7.82055172413793,0.7746206896551727) circle (0.5pt);\draw [fill=qqqqff] (8.243463203463204,1.7614141414141442) circle (0.5pt);
\draw [fill=qqqqff] (7.355781313090528,3.96212549462974) circle (0.5pt);
\end{scriptsize}
\end{tikzpicture}
\caption{An example illustrating the connections between ${\mathbb{F}_q}^2$ and $\mathbb{R}^2$; here $q=5$.}
\end{figure}
\end{example}

We can deduce from Theorem \ref{lienFR} the following result:

\begin{remark}\label{formt}
Let $F$ be an arrangement of $q+1$ lines in ${\mathbb{F}_q}^2$.\\
Let us assume that there exists an arrangement $A$ in $\mathbb{R}^2$ such that:
\begin{center}
$A\in\Phi([x_2(F),...,x_{|F|}(F)])$.
\end{center}
Then:
\begin{align*}
f_1(A)=(q+1)^2-x_1(F)\enspace\text{and}\enspace f_2(A)=2(q+1)+x_0(F).
\end{align*}
Indeed, it is sufficient to note that, according to Proposition \ref{form}, $f_1(A)=f_1^b(A)+2p_1(A)$ and $f_2(A)=f_2^b(A)+2p_1(A)$. Theorem \ref{lienFR} 
and the fact that $p_1(A)=q+1$ enable us to conclude.
\end{remark}

As a result, we can define in ${\mathbb{F}_q}^2$, in view of Remark \ref{formt}, some quantities, equivalent to those defined in $\mathbb{R}^2$, associated with 
an arrangement $F$ of $q+1$ lines:
\begin{enumerate}[label=$\bullet$]
 \item $\tilde{f}_0(F)=\tilde{p}_0(F)$;
 \item $\tilde{f}_1(F)=(q+1)^2-\tilde{x}_1(F)$;
 \item $\tilde{f}_2(F)=2(q+1)+\tilde{x}_0(F)$.
\end{enumerate}

Euler's formula applied to an arrangement $A$ in $\mathbb{R}^2$ gives us: $(f_0(A)+1)-f_1(A)+f_2(A)=\chi(S^2)=2$, where $\chi(S^2)$ is the 
Euler-Poincaré characteristic of the sphere $S^2$ (see \cite{Cartier}). In ${\mathbb{F}_q}^2$, we obtain similarly, for all arrangements $F$ of $q+1$ 
lines:

\begin{equation}\label{EP}
(\tilde{f}_0(F)+1)-\tilde{f}_1(F)+\tilde{f}_2(F)=2.
\end{equation}

Indeed, we have:
\begin{center}
$(\tilde{f}_0(F)+1)-\tilde{f}_1(F)+\tilde{f}_2(F)=(\sum_{i=2}^{q+1}\tilde{x}_i(F)+1)-(q^2+2q+1-\tilde{x}_1(F))+(2q+2+\tilde{x}_0(F))$, 
\end{center}
equality (\ref{equ1}) allowing us to obtain the expected result.

Extending the notation above to all sets of lines, we obtain that this equality is still valid for all arrangements in ${\mathbb{F}_q}^2$. Indeed the 
proof of 
equality (\ref{EP}) relies only on equality (\ref{equ1}), which is always true.

\section*{Conclusion}

To conclude, equalities and inequalities brought to light in the first section enlarge the knowledge on minimal Besicovitch arrangements. In 
addition, in the second section, the values of $x_0^B$ and $x_1^B$, for an arrangement of $q+1$ lines $B$ in ${\mathbb{F}_q}^2$, the $q$ elements 
finite field, emerge from Theorem \ref{lienFR} through two new equalities implying a well-chosen arrangement in $\mathbb{R}^2$. This first value is 
useful in the two-dimensional version of the finite Kakeya problema; indeed, it is the number of points in ${\mathbb{F}_q}^2$ through which at least one 
line of a minimal Besicovitch arrangement $B$ passes (easy to determine from the value of $x_0^B$) which is expected in this issue (see \cite{Faber,
Blokhuis}). The second value, for its part, is required in the determination of the complexity of self-dual normal bases (see \cite{PV,ste}). 
Therefore, all these equalities and inequalities open up new prospects for future research in this two specific fields.

%\section*{Acknowledgements}

\bibliography{bib}

\end{document}